\documentclass{amsart}

\usepackage[section]{placeins}

\usepackage{amsmath}             
\usepackage{amscd}
\usepackage{amssymb}             
\usepackage{amsfonts}            
\usepackage{latexsym}            
\usepackage{amsthm}              
\usepackage{graphicx}

\usepackage{enumerate}
\usepackage{mathrsfs} 
\usepackage{thmtools}
\usepackage{thm-restate}

\newcommand{\F}{\mathbb{F}}        

\newcommand{\SL}{\operatorname{SL}}
\newcommand{\GL}{\operatorname{GL}}

\newtheorem{lause}{Theorem}[section]

\newtheorem{lemma}[lause]{Lemma}
\newtheorem{seur}[lause]{Corollary}

\newtheorem*{lause*}{Theorem}

\theoremstyle{definition}

\theoremstyle{remark}
\newtheorem{remark}[lause]{Remark}
\newtheorem{esim}[lause]{Example} 
\newtheorem*{mot*}{Motivation}
\newtheorem*{acknow*}{Acknowledgements}

\numberwithin{equation}{section}

\usepackage{hyphenat}
\allowdisplaybreaks

\begin{document}

\title[Systems of imprimitivity]{Systems of imprimitivity for wreath products}

\author{Mikko Korhonen}
\address{Department of Mathematics, Southern University of Science and Technology, \text{Shenzhen} 518055, Guangdong, P. R. China}
\email{korhonen\_mikko@hotmail.com {\text{\rm(Korhonen)}}}
\thanks{}

\author{Cai Heng Li}
\address{SUSTech International Center for Mathematics and Department of Mathematics, Southern University of Science and Technology, Shenzhen 518055, Guangdong, \text{P.R. China}}
\email{lich@sustech.edu.cn {\text{\rm(Li)}}}
\thanks{Partially supported by NSFC grant 11931005.}


\date{\today}

\begin{abstract}

Let $G$ be an irreducible imprimitive subgroup of $\operatorname{GL}_n(\mathbb{F})$, where $\mathbb{F}$ is a field. Any system of imprimitivity for $G$ can be refined to a \emph{nonrefinable} system of imprimitivity, and we consider the question of when such a refinement is unique. Examples show that $G$ can have many nonrefinable systems of imprimitivity, and even the number of components is not uniquely determined. We consider the case where $G$ is the wreath product of an irreducible primitive $H \leq \operatorname{GL}_d(\mathbb{F})$ and transitive $K \leq S_k$, where $n = dk$. We show that $G$ has a unique nonrefinable system of imprimitivity, except in the following special case: $d = 1$, $n = k$ is even, $|H| = 2$, and $K$ is a subgroup of $C_2 \wr S_{n/2}$. As a simple application, we prove results about inclusions between wreath product subgroups.

\end{abstract}

\vspace*{-2ex}
\maketitle

\section{Introduction}

Let $G$ be an irreducible subgroup of $\GL(V)$, where $V$ is a finite-dimensional vector space over a field $\F$. We say that $G$ is \emph{imprimitive}, if there exists a decomposition $$V = W_1 \oplus \cdots \oplus W_k$$ with $k > 1$ such that $G$ acts on the set $\Gamma = \{W_1, \ldots, W_k\}$ of the summands $W_i$. In this case $\Gamma$ is called a \emph{system of imprimitivity} for $G$. A system of imprimitivity $\{Z_1, \ldots, Z_{\ell}\}$ is said to be a \emph{refinement} of $\Gamma$, if each $W_i$ is a direct sum of some $Z_j$'s. If no proper refinement of $\Gamma$ exists, we say that $\Gamma$ is \emph{nonrefinable}.

Is a nonrefinable system of imprimitivity of $G$ unique? Examples show that the answer is no in general. Even the number of summands in a nonrefinable system is not uniquely determined --- we provide examples of such behaviour in the next section.

Let $\Gamma = \{W_1, \ldots, W_k\}$ be a nonrefinable system of imprimitivity for $G$. Then it is a basic result \cite[Lemma 15.5]{SuprunenkoBook} that $G$ is conjugate to a subgroup of $N_G(W_1) \wr K$, where $K$ is the image of $G$ in the symmetric group $S_k$. Since $G$ is irreducible, it follows that $K$ is transitive and furthermore the action of $N_G(W_1)$ on $W_1$ is nontrivial, irreducible, and primitive \cite[Theorem 15.1, Lemma 15.4]{SuprunenkoBook}. In the case where $G$ is equal to such a wreath product, we have the following positive result which will be proven in this note.

\begin{lause}\label{thm:uniquesystemgeneralcase}Suppose that $n = dk$, where $k > 1$. Let $H$ be a nontrivial irreducible primitive subgroup of $\GL_d(\F)$ and let $K \leq S_k$ be transitive, so that the subgroup $G = H \wr K$ of $\GL_n(\F)$ is irreducible. Then $G$ has a unique nonrefinable system of imprimitivity, except when $d = 1$, $n = k$ is even, $|H| = 2$, and $K$ is a subgroup of $C_2 \wr S_{n/2}$.\end{lause}

A similar result was previously claimed in \cite[Theorem 2]{KonjuhVS}, but unfortunately the proof given there is based on a false result (Remark \ref{remark:notekonjuh}). Uniqueness for systems of imprimitivity has been considered by some authors in the context of finite complex reflection groups. See for example \cite[Lemma 2.7]{Cohen76} or \cite[Lemma 1.1]{KemperMalle}, which are related to Theorem \ref{thm:uniquesystemgeneralcase} in the case where $d = 1$, $K = S_n$, and $H$ is finite cyclic. The exceptional case of Theorem \ref{thm:uniquesystemgeneralcase} is also related to examples of wreath products where the base group is not a characteristic subgroup, see \cite[Theorem 5.1]{Gross88} and \cite[Theorem 9.12]{Neumann64}.

The proof of Theorem \ref{thm:uniquesystemgeneralcase} will be given in Section \ref{section:mainproof}. As a simple application, we prove results about maximal solvable subgroups of $\GL_n(\F)$ (Corollary \ref{cor:imprimitivemaximalsoluble}) and inclusions between wreath product subgroups (Corollary \ref{cor:application2}) in Section \ref{section:applications}.

\section{Examples of nonuniqueness}

In general a nonrefinable system of imprimitivity $\Gamma$ is not unique for $G$, and an infinite family of examples is provided by the exception in Theorem \ref{thm:uniquesystemgeneralcase} (see Remark \ref{remark:exception} in the next section). In this family of examples, the number of components in a nonrefinable system of imprimitivity is uniquely determined. It turns out that it is also possible for $G$ to have nonrefinable systems of imprimitivity with different numbers of components. The following provides the smallest possible examples.

\begin{esim}\label{esim:examplesystems}Let $G = \GL_2(3)$ and let $q$ be a prime power such that $q \equiv 1 \mod{6}$. Then one can embed $G \leq \GL_4(q)$ such that for $V = \F_q^4$, we have: 

\begin{enumerate}[\normalfont (i)]
\item $G$ is irreducible;
\item $V = Z_1 \oplus Z_2 \oplus Z_3 \oplus Z_4$, such that $\dim Z_i = 1$ and $G$ acts on $\{ Z_1, Z_2, Z_3, Z_4\}$;
\item $V = W_1 \oplus W_2$, such that $\dim W_i = 2$ and $G$ acts on $\{W_1, W_2\}$;
\item Both systems of imprimitivity in (ii) and (iii) are nonrefinable.
\end{enumerate}
\end{esim}

\begin{proof}Let $x, y \in G$ be as follows: \begin{align*}x &= \begin{pmatrix} 1 & 0 \\ 0 & -1 \end{pmatrix}, & y &= \begin{pmatrix} -1 & 1 \\ 0 & -1 \end{pmatrix}. \end{align*} Then $K = \langle x, y \rangle \cong D_{12}$, with $K / [K,K] = \langle \overline{x},\ \overline{y} \rangle \cong C_2 \times C_2$. Let $W$ be the $1$-dimensional $\F_q[K]$-module corresponding to the linear character $\theta: K \rightarrow \F_q^\times$ such that $\theta(x) = 1$ and $\theta(y) = -1$. Consider the induced $\F_q[G]$-module $V = \operatorname{Ind}_{K}^G(W)$. We have $\dim V = [G:K] = 4$, and a calculation shows that $V$ is a faithful irreducible $\F_q[G]$-module, so claim (i) holds. Since we are inducing a $1$-dimensional module, it is clear that we get a decomposition $V = Z_1 \oplus Z_2 \oplus Z_3 \oplus Z_4$ as in (ii), which is nonrefinable since $\dim Z_i = 1$.

Let $H = \SL_2(3)$, so $[G:H] = 2$ and $H \trianglelefteq G$. Note that by Maschke's theorem $\F_q[H]$ is semisimple. Thus by examining the ordinary character table of $H$, we can see that $\F_q$ is a splitting field for $H$, since it contains a primitive cube root of unity. Then by looking at the character degrees, we conclude that there is no irreducible $\F_q[H]$-module of dimension $4$. 

In particular, the restriction of $V$ to $H$ is not irreducible. Thus by Clifford theory, the restriction decomposes as $$V = W_1 \oplus W_2,$$ where $W_1$, $W_2$ are non-isomorphic irreducible $\F_q[H]$-modules with $\dim W_i = 2$. Then $G$ acts on $\{W_1, W_2\}$ and claim (iii) holds. 

What remains is to check that $V = W_1 \oplus W_2$ provides a nonrefinable system of imprimitivity for $G$. Equivalently, we need to check that the action of $H$ on $W_1$ is primitive, but this is immediate from the fact that $H$ does not have a subgroup of index $2$.\end{proof}

\begin{remark}\label{remark:notekonjuh}The paper \cite{KonjuhVS} claims in its main theorem that for an irreducible imprimitive subgroup of $\GL(V)$, the number of components in a nonrefinable system of imprimitivity is unique. Example \ref{esim:examplesystems} shows that the claim is false, and the mistake in \cite{KonjuhVS} is on p. 6, line 7: the author argues that $N = N_1 \oplus \cdots \oplus N_\ell$ since $N_i \cap N_j = 0$ for $i \neq j$ (which is in general false, unless $\ell = 2$.)\end{remark}

\section{Systems of imprimitivity}\label{section:mainproof}

In this section, we will prove Theorem \ref{thm:uniquesystemgeneralcase}. We first need two lemmas. The first one of these is well known and not difficult to prove, so we will omit the proof.

\begin{lemma}\label{lemma:submodulesofmfree}
Let $M$ be a group and suppose that $V$ is a completely reducible $\F[M]$-module such that $V = W_1 \oplus \cdots \oplus W_k$, where $W_1$, $\ldots$, $W_k$ are irreducible and pairwise non-isomorphic $\F[M]$-modules. Then any nonzero $\F[M]$-submodule of $V$ is of the form $W_{i_1} \oplus \cdots \oplus W_{i_{\alpha}}$, for some $\alpha > 0$ and $1 \leq i_1 < \cdots < i_{\alpha} \leq k$.
\end{lemma}

\begin{lemma}\label{lemma:directprodcrimpr}

Let $M = H_1 \times \cdots \times H_k$ be a group, and let $V$ be an $\F[M]$-module such that the following hold: 
	\begin{enumerate}[\normalfont (i)]
	 \item $V = W_1 \oplus \cdots \oplus W_k$, where $W_i$ is a nontrivial irreducible $\F[H_i]$-module for all $1 \leq i \leq k$;
	 \item the action of $H_i$ on $W_i$ is primitive for all $1 \leq i \leq k$; and
	 \item the direct factors $H_j$ act trivially on $W_i$ for all $j \neq i$.  
 \end{enumerate}

If $V = Q_1 \oplus \cdots \oplus Q_\ell$ and $M$ acts on $\{Q_1, \ldots, Q_\ell\}$, then we have $\ell \leq k$.
\end{lemma}

\begin{proof}
Note that the $W_i$ are irreducible and pairwise non-isomorphic $\F[M]$-modules, so by Lemma \ref{lemma:submodulesofmfree} any $\F[M]$-submodule of $V$ is a direct sum $W_{i_1} \oplus \cdots \oplus W_{i_\alpha}$ for some $1 \leq i_1 < \cdots < i_{\alpha} \leq k$. (We will use this fact throughout the proof.)

For the proof of the lemma, we proceed by induction on $k$. In the case $k = 1$, if $V = Q_1 \oplus \cdots \oplus Q_\ell$ and $M$ acts on $\{Q_1, \ldots, Q_\ell\}$, then $\ell = 1$ since $M = H_1$ acts primitively on $V = W_1$. Suppose then that $k > 1$.

Consider first the case where $M$ is not transitive on $\{Q_1, \ldots, Q_\ell\}$. Let $$\{Q_1^{(1)}, \ldots, Q_{d_1}^{(1)}\}, \ldots, \{Q_1^{(s)}, \ldots, Q_{d_s}^{(s)}\}$$ be the orbits of $M$ on $\{Q_1, \ldots, Q_\ell\}$. Then $$V = (Q_1^{(1)} \oplus \cdots \oplus Q_{d_1}^{(1)}) \oplus \cdots \oplus (Q_1^{(s)} \oplus \cdots \oplus Q_{d_s}^{(s)})$$ where by Lemma \ref{lemma:submodulesofmfree}, for all $1 \leq i \leq s$ we have $$Q_1^{(i)} \oplus \cdots \oplus Q_{d_i}^{(i)} = W_{1}^{(i)} \oplus \cdots \oplus W_{\alpha_i}^{(i)}$$ for some subset $\{W_{1}^{(i)}, \ldots, W_{\alpha_i}^{(i)} \}$ of $\{W_1, \ldots, W_k\}$. Now the action of the direct product $H_{1}^{(i)} \times \cdots \times H_{\alpha_i}^{(i)}$ on $W_{1}^{(i)} \oplus \cdots \oplus W_{\alpha_i}^{(i)}$ satisfies conditions (i) -- (iii) of the lemma, so $d_i \leq \alpha_i$ for all $i$ by induction. Since \begin{align*} \ell &= d_1 + \cdots + d_s, \\ k &= \alpha_1 + \cdots + \alpha_s, \end{align*} we conclude that $\ell \leq k$.

Thus we can assume that $M$ acts transitively on $\{Q_1, \ldots, Q_\ell\}$. Let $k_0 > 0$ be minimal such that $$Q_j \cap (W_{j_1} \oplus \cdots \oplus W_{j_{k_0}}) \neq 0$$ for some $1 \leq j \leq \ell$ and $1 \leq j_1 < \cdots < j_{k_0} \leq k$. 

For $1 \leq i \leq \ell$, set $Q_i' := Q_i \cap (W_{j_1} \oplus \cdots \oplus W_{j_{k_0}})$. Then $M$ acts on $\{Q_1', \ldots, Q_{\ell}'\}$, so $Q_1' \oplus \cdots \oplus Q_{\ell}'$ is an $\F[M]$-submodule of $W_{j_1} \oplus \cdots \oplus W_{j_{k_0}}$. By Lemma \ref{lemma:submodulesofmfree} and the minimality of $k_0$, we have in fact $$W_{j_1} \oplus \cdots \oplus W_{j_{k_0}} = Q_1' \oplus \cdots \oplus Q_{\ell}'.$$ If $k_0 < k$, then by induction we have $\ell \leq k_0$ and so $\ell < k$. Thus we can assume that $k_0 = k$, so \begin{equation}\label{eq:zerointersect}Q_i \cap (W_{j_1} \oplus \cdots \oplus W_{j_{k-1}}) = 0\end{equation} for all $1 \leq i \leq \ell$ and $1 \leq j_1 < \cdots < j_{k-1} \leq k$. In particular, the projection of $Q_j$ into any $W_i$ is injective, so \begin{equation}\label{eq:firstineq}\dim Q_j \leq \dim W_i\end{equation} for all $i$ and $j$.

Next let $s > 0$ be minimal such that $$W_i \cap (Q_{i_1} \oplus \cdots \oplus Q_{i_{s}}) \neq 0$$ for some $1 \leq i \leq k$ and $1 \leq i_1 < \cdots < i_{s} \leq \ell$. 

We will show that for all $j \neq i$, the subgroup $H_j$ acts on $Q_{i_1} \oplus \cdots \oplus Q_{i_{s}}$ nontrivially. Let $h \in H_j$. We have $h(Q_{i_1} \oplus \cdots \oplus Q_{i_s}) = Q_{i_1'} \oplus \cdots \oplus Q_{i_s'}$ for some $1 \leq i_1' < \cdots < i_{s}' \leq \ell$. Since $H_j$ acts trivially on $W_i$, it follows that $$W_i \cap (Q_{i_1} \oplus \cdots \oplus Q_{i_{s}}) = W_i \cap (Q_{i_1} \oplus \cdots \oplus Q_{i_{s}}) \cap (Q_{i_1'} \oplus \cdots \oplus Q_{i_s'}).$$ Thus $Q_{i_1} \oplus \cdots \oplus Q_{i_{s}} = Q_{i_1'} \oplus \cdots \oplus Q_{i_s'}$ by the minimality of $s$, so $H_j$ acts on $Q_{i_1} \oplus \cdots \oplus Q_{i_{s}}$.

To see that the action is nontrivial, let $v \in Q_{i_1}$ be nonzero. By~\eqref{eq:zerointersect}, we have $v = w_1 + \cdots + w_k$ where $w_r \in W_r$ and $w_r \neq 0$ for all $1 \leq r \leq k$. Since $W_j$ is a nontrivial irreducible $\F[H_j]$-module, we have $gw_j \neq w_j$ for some $g \in H_j$. Then $gv \neq v$, so $H_j$ acts nontrivially on $Q_{i_1} \oplus \cdots \oplus Q_{i_{s}}$. Consequently $W_j$ must be contained in $Q_{i_1} \oplus \cdots \oplus Q_{i_{s}}$. In particular $W_j \cap (Q_{i_1} \oplus \cdots \oplus Q_{i_{s}}) \neq 0$, so by repeating the same arguments we conclude that $W_i$ is also contained in $Q_{i_1} \oplus \cdots \oplus Q_{i_{s}}$.

Therefore $Q_{i_1} \oplus \cdots \oplus Q_{i_{s}} = W_1 \oplus \cdots \oplus W_k$, so $s = \ell$ and $W_i \cap (Q_{i_1} \oplus \cdots \oplus Q_{i_{\ell-1}}) = 0$ for all $i$ and $1 \leq i_1 < \cdots < i_{\ell-1} \leq \ell$. Hence the projection of $W_i$ into any $Q_j$ is injective, so $\dim W_i \leq \dim Q_j$ for all $i$ and $j$.  By~\eqref{eq:firstineq} we conclude that $\dim Q_j = \dim W_i$ for all $i$ and $j$, from which it follows that $\ell = k$. This completes the proof of the lemma.\end{proof}

\begin{proof}[Proof of Theorem \ref{thm:uniquesystemgeneralcase}.]
Let $\{W_1, \ldots, W_k\}$ be the system of imprimitivity defining $G$. Then $V = W_1 \oplus \cdots \oplus W_k$ and $G = (H_1 \times \cdots \times H_k) \rtimes K$, where the action of $H_i$ is nontrivial irreducible primitive on $W_i$, and trivial on $W_j$ for $j \neq i$. Furthermore, the action of $K$ on $\{W_1, \ldots, W_k\}$ is faithful and transitive. We denote the base group $H_1 \times \cdots \times H_k$ by $M$.

Suppose that there is another nonrefinable system of imprimitivity, say $V = Z_1 \oplus \cdots \oplus Z_\ell$ such that $\ell > 1$ and $G$ acts on $\{Z_1, \ldots, Z_\ell\}$. Since $G$ is irreducible, the action on $\{Z_1, \ldots, Z_\ell\}$ must be transitive. Furthermore, the action of $N_G(Z_i)$ on $Z_i$ must be irreducible and primitive \cite[Theorem 15.1]{SuprunenkoBook}.

Let $s > 0$ be minimal such that $Z_i \cap (W_{j_1} \oplus \cdots \oplus W_{j_s}) \neq 0$ for some $1 \leq i \leq \ell$ and $1 \leq j_1 < \cdots < j_s \leq k$. First consider the case where $s = 1$, so $Z_i \cap W_j \neq 0$ for some $i$ and $j$. Then $$(Z_1 \cap W_j) \oplus \cdots \oplus (Z_\ell \cap W_j)$$ is a non-zero $N_G(W_j)$-submodule of $W_j$. Since $N_G(W_j)$ acts irreducibly on $W_j$, we have $W_j = (Z_1 \cap W_j) \oplus \cdots \oplus (Z_\ell \cap W_j)$. Furthermore, the action of $N_G(W_j)$ is primitive, so $W_j = Z_i \cap W_j$. Repeating this argument for $Z_i$, we see that $Z_i = Z_i \cap W_j$, so $Z_i = W_j$ and $\{W_1, \ldots, W_k\} = \{Z_1, \ldots, Z_\ell\}$.

Therefore we can suppose that $s > 1$ in what follows. Let $\{Z_{i_1}, \ldots, Z_{i_r}\}$ be the orbit of $Z_i$ under the base group $M$. For $1 \leq t \leq r$, set $$Q_t := Z_{i_t} \cap (W_{j_1} \oplus \cdots \oplus W_{j_s}).$$ Then $Q_1 \oplus \cdots \oplus Q_r$ is an $\F[M]$-submodule of $W_{j_1} \oplus \cdots \oplus W_{j_s}$, so by Lemma \ref{lemma:submodulesofmfree} and the minimality of $s$ we conclude that $$W_{j_1} \oplus \cdots \oplus W_{j_s} = Q_1 \oplus \cdots \oplus Q_r.$$
    
Let $v \in Z_i \cap (W_{j_1} \oplus \cdots \oplus W_{j_s})$ be non-zero, and write $v = w_1 + \cdots + w_s$, where $w_t \in W_{j_t}$. Note that each $w_t$ is non-zero by the minimality of $s$. For $h_t \in H_{j_t}$ ($1 \leq t \leq s-1$), we define $$v_{h_1, \ldots, h_{s-1}} := h_1 w_1 + \cdots + h_{s-1} w_{s-1} + w_s.$$ Since $v \in Z_i$ and $M$ acts on $W_{j_1} \oplus \cdots \oplus W_{j_s} \subseteq Z_{i_1} \oplus \cdots \oplus Z_{i_r}$, each $v_{h_1, \ldots, h_{s-1}}$ is contained in some $Z_{i_t}$. 

We claim that $v_{h_1, \ldots, h_{s-1}}$ and $v_{h_1', \ldots, h_{s-1}'}$ can be contained in the same $Z_{i_t}$ only if they are equal. Indeed, if $v_{h_1, \ldots, h_{s-1}}$ and $v_{h_1', \ldots, h_{s-1}'}$ are both contained in $Z_{i_t}$, then $$v_{h_1, \ldots, h_{s-1}} - v_{h_1', \ldots, h_{s-1}'} = (h_1 - h_1')w_1 + \cdots + (h_{s-1} - h_{s-1}')w_{s-1}$$ is contained in $Z_{i_t} \cap (W_{j_1} \oplus \cdots \oplus W_{j_{s-1}})$, and thus must be zero by the minimality of $s$. 
		
It follows then that $r \geq |\Pi_1| \cdots |\Pi_{s-1}|$, where $\Pi_t$ is the $H_{j_t}$-orbit of $w_t$. Note that $|\Pi_t| \geq 2$ for all $1 \leq t \leq s-1$, since each $w_t$ is nonzero, and since $H_{j_t}$ acts nontrivially. Furthermore, we have $r \leq s$ by Lemma \ref{lemma:directprodcrimpr}, so $$s \geq r \geq |\Pi_1| \cdots |\Pi_{s-1}| \geq 2^{s-1},$$ which forces $s = r = 2$ and $|\Pi_1| = 2$.

Write $\Pi_1 = \{w_1, w_1'\}$. Then $w_1 + w_1'$ is fixed by the action of $H_{j_1}$ and thus $w_1 + w_1' = 0$. Hence $hw_1 = \pm w_1$ for all $h \in H_{j_1}$. We conclude then from the irreducibility of $H_{j_1}$ that $\dim W_j = 1$ for all $j$, and furthermore $|H| = 2$. Note that this also forces $\operatorname{char} \F \neq 2$.

To complete the proof of the theorem, it remains to show that $n$ is even and $K \leq C_2 \wr S_{n/2}$. To this end, we first adapt an argument from \cite[p. 276]{ShephardTodd54} to show that $\dim Z_{i} = 1$. Suppose, for the sake of contradiction, that $\dim Z_{i} > 1$. Let $h \in H_{j_1}$ be such that $hw_1 = -w_1$. Since $h$ acts trivially on $W_t$ for $t \neq j_1$, it follows that the fixed point space $V^h$ has dimension $n-1$. Thus $Z_{i}$ has nonzero intersection with $V^h$, which implies that $h Z_{i} = Z_{i}$ since $G$ acts on $\{Z_1, \ldots, Z_\ell\}$. So then both $v = w_1 + w_2$ and $hv = -w_1 + w_2$ would be contained in $Z_{i}$, which implies that $v + hv = 2w_2 \in Z_{i}$. Thus $w_2 \in Z_{i}$ since $\operatorname{char} \F \neq 2$, so we have a contradiction due to $Z_{i} \cap W_{j_2} = 0$.

Therefore we have $\ell = k = n$ and $\dim Z_j = 1$ for all $1 \leq j \leq n$. Note that now $W_{j_1} \oplus W_{j_2} = Z_{i_1} \oplus Z_{i_2}$, and $\{Z_{i_1}, Z_{i_2}\}$ is an orbit for the action of $M$ on $\{Z_1, \ldots, Z_n\}$. Since $M$ is a normal subgroup of $G$ and since $G$ acts transitively on $\{Z_1, \ldots, Z_n\}$, every $M$-orbit is of order $2$, and so $n$ is even. By relabeling the summands if necessary, we can assume that the $M$-orbits are $$\{Z_1, Z_2\},\ \ldots,\ \{Z_{n-1}, Z_n\}$$ and furthermore that we have $$W_1 \oplus W_2 = Z_1 \oplus Z_2,\ \ldots,\ W_{n-1} \oplus W_n = Z_{n-1} \oplus Z_n.$$ Thus $G$ acts on the set of pairs $\{ \{W_1, W_2\}, \ldots, \{W_{n-1}, W_n\} \}$, which shows that $K \leq C_2 \wr S_{n/2}$.\end{proof}

\begin{remark}\label{remark:exception}The exception in Theorem \ref{thm:uniquesystemgeneralcase} is a genuine exception. In this case $H$ is cyclic of order $2$, so $H = \{ \pm 1\} \leq \GL_1(\F)$ and $\operatorname{char} \F \neq 2$. We can write $G = (H_1 \times \cdots \times  H_n) \rtimes K,$ where $H_i = \langle \sigma_i \rangle$ is cyclic of order $2$. Furthermore $n$ is even, and $K$ is a transitive subgroup of $C_2 \wr S_{n/2}$. Thus we can find a basis $\{e_1, \ldots, e_n \}$ of $V = \F^n$ such that $G$ acts as follows: \begin{align*} \sigma_i(e_i) &= -e_i & \text{for all } i, \\ \sigma_i(e_j) &= e_j & \text{for all } i \neq j, \\ \pi e_i &= e_{\pi(i)} & \text{for all } \pi \in K. \end{align*} Moreover we can assume that $K$ acts on the pairs $\{ \{e_1,e_2 \}, \ldots, \{e_{n-1},e_n \} \}$. Now $\{ \langle e_1 \rangle, \ldots, \langle e_n \rangle \}$ is the system of imprimitivity that defines $G$, and it is clear from the action that the decomposition $$V = \langle e_1 + e_2 \rangle \oplus \langle e_1 - e_2 \rangle \oplus \cdots \oplus \langle e_{n-1} + e_n \rangle \oplus \langle e_{n-1} - e_n \rangle$$ provides another system of imprimitivity for $G$. If there exists $\lambda \in \F$ with $\lambda^2 = -1$, then $$V = \langle e_1 + \lambda e_2 \rangle \oplus \langle e_1 - \lambda e_2 \rangle \oplus \cdots \oplus \langle e_{n-1} + \lambda e_n \rangle \oplus \langle e_{n-1} - \lambda e_n \rangle$$ gives also a system of imprimitivity for $G$. For $n = 2$, these examples appear in \cite[Remark 2.8]{Cohen76}.

With a few more arguments, we can describe all systems of imprimitivity for $G$. The proof of Theorem \ref{thm:uniquesystemgeneralcase} shows that any system of imprimitivity distinct from $\{ \langle e_1 \rangle, \ldots, \langle e_n \rangle \}$ must be of the form $\{Z_1, \ldots, Z_n\}$, where $\dim Z_i = 1$ for all $1 \leq i \leq n$ and $Z_i \cap \langle e_j \rangle = 0$ for all $1 \leq i,j \leq n$. Furthermore, the action of $K$ on $\{e_1, \ldots, e_n\}$ has a system of imprimitivity $\{ \{ f_1, f_2 \}, \ldots, \{ f_{n-1}, f_n \} \}$ (possibly different from $\{ \{ e_{1}, e_{2} \}, \ldots, \{ e_{n-1}, e_{n} \} \}$) such that $$Z_1 \oplus Z_2 = \langle f_1 \rangle \oplus \langle f_2 \rangle,\ \ldots,\ Z_{n-1} \oplus Z_n = \langle f_{n-1} \rangle \oplus \langle f_n \rangle.$$ Therefore $Z_1 = \langle f_1 + \lambda f_2 \rangle$ and $Z_2 = \langle f_{1} + \mu f_{2} \rangle$ for some $\lambda, \mu \in \F \setminus \{0\}$. An element of $G$ for which $f_{1} \mapsto f_{1}$ and $f_{2} \mapsto -f_{2}$ must map $Z_1$ to $Z_2$, so we conclude that $\mu = -\lambda$. Since $K$ acts transitively, there exists an element of $G$ which swaps $f_1$ and $f_2$. Such an element acts on $\{Z_1, Z_2\}$ and maps $f_1 + \lambda f_2$ to $\lambda (f_1 + \lambda^{-1} f_2)$, so $\lambda^{-1} = \lambda$ or $\lambda^{-1} = -\lambda$. Furthermore, the action of $K$ on the pairs $\{f_i, f_{i+1}\}$ is transitive, so $Z_i = \langle f_i \pm \lambda f_{i+1} \rangle$ and $Z_{i+1} = \langle f_i \mp \lambda f_{i+1} \rangle$ for all $1 \leq i < n$ odd.

We conclude then that any system of imprimitivity distinct from $\{ \langle e_1 \rangle, \ldots, \langle e_n \rangle \}$ corresponds to a decomposition $$V = \langle f_1 + \lambda f_2 \rangle \oplus \langle f_1 - \lambda f_2 \rangle \oplus \cdots \oplus \langle f_{n-1} + \lambda f_n \rangle \oplus \langle f_{n-1} - \lambda f_n \rangle,$$ where $\{ \{ f_1, f_2 \}, \ldots, \{ f_{n-1}, f_n \} \}$ is some system of imprimitivity for the action of $K$ on $\{e_1, \ldots, e_n\}$, and $\lambda \in \F$ is such that $\lambda^2 = \pm 1$.\end{remark}

\section{Applications}\label{section:applications}

Our original motivation for Theorem \ref{thm:uniquesystemgeneralcase} was in the problem of classifying maximal irreducible solvable subgroups of $\GL_n(\F)$. It follows from \cite[Theorem 15.4]{SuprunenkoBook} that if $G \leq \GL_n(\F)$ is maximal irreducible solvable, then either:

\begin{enumerate}[(1)] 
\item $G$ is primitive; or
\item $n = dk$ for $k > 1$, and $G = H \wr K$, where $H \leq \GL_d(\F)$ is maximal irreducible primitive solvable and $K \leq S_k$ is maximal transitive solvable.
\end{enumerate}

Note that the groups in case (2) are not always maximal solvable. For example, the imprimitive subgroup $\GL_1(q) \wr C_2$ is not maximal solvable in $\GL_2(q)$ if $q = 3$ or $q = 5$. When are they maximal solvable?  As a corollary of Theorem \ref{thm:uniquesystemgeneralcase}, we can reduce this question to the problem of determining when such $H \wr K$ lie in a primitive solvable subgroup of $\GL_n(\F)$.

\begin{seur}\label{cor:imprimitivemaximalsoluble}
Suppose that $n = dk$ with $k > 1$. Let $G = H \wr K \leq \GL_n(\F)$, where $H \leq \GL_d(\F)$ is maximal irreducible primitive solvable and $K \leq S_k$ is maximal solvable transitive. Then the following statements are equivalent:

\begin{enumerate}[\normalfont (i)]
\item $G$ is not maximal solvable in $\GL_n(\F)$.
\item $k = \ell' \ell$ for some $\ell' > 1$ such that $K = X \wr Y$, where $X \leq S_{\ell'}$ and $Y \leq S_{\ell}$ are maximal transitive solvable, and $H \wr X$ is contained in a maximal irreducible primitive solvable subgroup of $\GL_{d\ell'}(\F)$.
\end{enumerate}

\end{seur}

\begin{proof}
If (ii) holds, then $G$ is not maximal solvable, since $G = H \wr (X \wr Y) = (H \wr X) \wr Y < H_0 \wr Y$ for some maximal irreducible primitive solvable subgroup $H_0$ of $\GL_{d \ell'}(\F)$. For the other direction, suppose that $G$ is not maximal solvable. By a theorem of Zassenhaus \cite[Satz 8]{Zassenhaus37}, there exists a maximal solvable subgroup $G_0 \leq \GL_n(\F)$ that contains $G$. 

If $G_0$ is primitive, then (ii) holds with $X = K$ and $Y = 1$. Suppose then that $G_0$ is imprimitive. In this case, by \cite[Theorem 15.4]{SuprunenkoBook} we have $G_0 = H_0 \wr K_0$ for some $H_0 \leq \GL_e(\F)$ maximal irreducible primitive solvable and $K_0 \leq S_{\ell}$ maximal transitive solvable, where $n = e \ell$ for $\ell > 1$. 

We assume first that $G$ is not as in the exceptional case of Theorem \ref{thm:uniquesystemgeneralcase}. Write $\F^n = W_1 \oplus \cdots \oplus W_k = Z_1 \oplus \cdots \oplus Z_{\ell},$ where $\{W_1, \ldots, W_k\}$ and $\{Z_1, \ldots, Z_\ell\}$ are the systems of imprimitivity defining $G$ and $G_0$, respectively. Applying Theorem \ref{thm:uniquesystemgeneralcase} to $G$, it follows that $\{W_1, \ldots, W_k\}$ must be a refinement of $\{Z_1, \ldots, Z_\ell\}$. In other words, we conclude that $\ell$ divides $k$, and for all $1 \leq i \leq \ell$ we have $$Z_i = W_1^{(i)} \oplus \cdots \oplus W_{k/\ell}^{(i)}$$ for some subset $B_i := \{W_1^{(i)}, \ldots, W_{k/\ell}^{(i)}\}$ of $\{W_1, \ldots, W_k\}$.

Therefore the sets $\{B_1, \ldots, B_{\ell} \}$ form a block system for the action of $K$ on $\{W_1,\ldots,W_k\}$, so $K$ is a subgroup of $X \wr Y$, where $X \leq S_{k/\ell}$ is the action of $N_{K}(B_1)$ on $B_1$, and $Y \leq K_0$ is the action of $K$ on $\{Z_1, \ldots, Z_{\ell}\}$. Furthermore, in this case we have $H \wr X \leq H_0$. By the maximality of $K$ we must have $K = X \wr Y$ with $X$ and $Y$ maximal transitive solvable, so (ii) holds.

What remains then is to consider the exceptional case of Theorem \ref{thm:uniquesystemgeneralcase}, in which case $n = k$, $d = 1$, and $H \leq \GL_1(\F)$ is cyclic of order $2$. Furthermore, in this case $n$ is even and $K$ is a transitive subgroup of $C_2 \wr S_{n/2}$. Since $H$ is assumed to be maximal solvable, we have $H = \GL_1(\F)$, so $\F = \F_3$ and $H = \{ \pm 1 \}$. Now $K$ normalizes the elementary abelian base group $C_2^{n/2}$ of $C_2 \wr S_{n/2}$, so by maximality $K$ must contain $C_2^{n/2}$. Thus $K = C_2 \wr T$ for some maximal transitive solvable subgroup $T$ of $S_{n/2}$. Since $\GL_2(\F) = \GL_2(3)$ is solvable, we conclude that (ii) holds with $X = C_2$ and $Y = T$.

\end{proof}

With similar arguments, we can apply Theorem \ref{thm:uniquesystemgeneralcase} to the problem of describing the inclusions between irreducible wreath product subgroups $H_1 \wr K_1$ of $\GL_n(\F)$, where $H_1$ is primitive. The following corollary of Theorem \ref{thm:uniquesystemgeneralcase} provides a solution in most cases.

\begin{seur}\label{cor:application2}
Suppose that $n = dk$, where $k > 1$. Let $G_1 = H_1 \wr K_1 \leq \GL_n(\F)$, where $H_1 \leq \GL_d(\F)$ is nontrivial irreducible primitive and $K_1 \leq S_k$ is transitive. Suppose that $G_1$ is not one of the exceptions of Theorem \ref{thm:uniquesystemgeneralcase}. Then $G_1$ is contained in an imprimitive subgroup $H_2 \wr K_2$ of $\GL_n(\F)$ if and only if all of the following conditions hold:

\begin{enumerate}[\normalfont (i)]
\item $n = e \ell$, $H_2 \leq \GL_e(\F)$ and $K_2 \leq S_{\ell}$ with $\ell > 1$ dividing $k$;
\item $K_1 \leq X \wr Y$, where $X \leq S_{k / \ell}$ and $Y \leq K_2$;
\item $H_1 \wr X \leq H_2$.
\end{enumerate}
\end{seur}

\begin{proof}If conditions (i) -- (iii) hold, it is clear that $H_1 \wr K_1 \leq H_1 \wr (X \wr Y) = (H_1 \wr X) \wr Y \leq H_2 \wr K_2.$ The other direction of the claim follows from Theorem \ref{thm:uniquesystemgeneralcase}, by arguing as in the proof of Corollary \ref{cor:imprimitivemaximalsoluble} (paragraphs 3--4).\end{proof}

What about when $G_1 = H_1 \wr K_1 \leq \GL_n(\F)$ is as in the exception of Theorem \ref{thm:uniquesystemgeneralcase}? In this case we know all the systems of imprimitivity for $G_1$ (Remark \ref{remark:exception}), which readily gives a description of the wreath product subgroups of $\GL_n(\F)$ that contain $G_1$.

\end{document}